\documentclass[11pt]{amsart}
\usepackage[T1]{fontenc}
\usepackage[latin1]{inputenc}
\usepackage{a4wide}
\usepackage{setspace}
\usepackage{xypic}
\usepackage{amssymb}
\usepackage{amsthm}
\usepackage[english]{babel}
\usepackage{latexsym}
\date{}

\newtheorem{lem}{Lemma}

\newtheorem*{thm*}{Theorem}

\newenvironment{f-proof}[1][\sc D\'emonstration.]{\begin{trivlist}
\item[\hskip \labelsep {\bfseries #1}]}{\hfill{$\square$}\end{trivlist}}

\newcommand{\Prod}{\displaystyle\prod}

\begin{document}

\title{Milne's correcting factor and derived de Rham cohomology}
\author{Baptiste Morin}

\maketitle

\begin{abstract}
Milne's correcting factor is a numerical invariant playing an important role in formulas for special values of zeta functions of varieties over finite fields. We show that Milne's factor is simply the Euler characteristic of the derived de Rham complex (relative to $\mathbb{Z}$) modulo the Hodge filtration.
\end{abstract}

\vspace{1cm}

A Result of Milne (\cite{Milne86} Theorem 0.1) describes the special values of the zeta function of a smooth projective variety over a finite field satisfying the Tate conjecture. This result was later reformulated by Lichtenbaum and Geisser (see \cite{Geisser-Weiletale}, \cite{Lichtenbaum-finite-field} and \cite{Milne13}) as follows. 
They conjecture that
\begin{equation}\label{LG-Conj}
\mathrm{lim}_{t\rightarrow q^{-n}} Z(X,t)\cdot (1-q^{n}t)^{\rho_n}= \pm \chi(H_W^*(X,\mathbb{Z}(n)),\cup e)\cdot q^{\chi(X/\mathbb{F}_q,\mathcal{O}_X,n)}
\end{equation}
and show that (\ref{LG-Conj}) holds whenever the groups $H_W^i(X,\mathbb{Z}(n))$ are finitely generated.
Here $H_W^*(X,\mathbb{Z}(n))$ denotes Weil-\'etale motivic cohomology,  $e\in H^1(W_{\mathbb{F}_q},\mathbb{Z})$ is a fundamental class and $\chi(H_W^*(X,\mathbb{Z}(n)),e)$ is the Euler characteristic of the complex
\begin{equation}\label{LG-complex}
...\stackrel{\cup e}{\longrightarrow} H^i_W(X,\mathbb{Z}(n)) \stackrel{\cup e}{\longrightarrow} H^{i+1}_W(X,\mathbb{Z}(n)) \stackrel{\cup e}{\longrightarrow}... 
\end{equation}
More precisely, the cohomology groups of the complex (\ref{LG-complex}) are finite and $\chi(H_W^*(X,\mathbb{Z}(n)),\cup e)$ is the alternating product of their orders. Finally, Milne's correcting factor $q^{\chi(X/\mathbb{F}_q,\mathcal{O},n)}$ was defined in\cite{Milne86} by the formula
$$\chi(X/\mathbb{F}_q,\mathcal{O}_X,n)=\sum _{i\leq n,j} (-1)^{i+j}\cdot (n-i)\cdot \mathrm{dim}_{\mathbb{F}_q} H^j(X,\Omega^i_{X/\mathbb{F}_q}).$$
It is possible to generalize (\ref{LG-Conj}) in order to give a conjectural description of special values of zeta functions of all separated schemes of finite type over $\mathbb{F}_q$ (see \cite{Geisser-arith-coh} Conjecture 1.4), and even of all motivic complexes over $\mathbb{F}_q$ (see \cite{Milne-Ramachandran13} Conjecture 1.2). The statement of those more general conjectures is in any case very similar to formula (\ref{LG-Conj}). The present note is motivated by the hope for a further generalization, which would apply to zeta functions of all algebraic schemes (and ultimately motives) over $\mathrm{Spec}(\mathbb{Z})$ (see \cite{Morin14} for the case $n=0$). As briefly explained below, the special-value conjecture for (flat) schemes over $\mathrm{Spec}(\mathbb{Z})$ must take a rather different form than formula (\ref{LG-Conj}). Going back to the special case of smooth projective varieties over finite fields, this leads to a slightly different restatement of formula (\ref{LG-Conj}).

Let $\mathcal{X}$ be a regular  scheme proper over $\mathrm{Spec}(\mathbb{Z})$. The "fundamental line"
$$\Delta(\mathcal{X}/\mathbb{Z},n):=\mathrm{det}_{\mathbb{Z}}R\Gamma_{W,c}(\mathcal{X},\mathbb{Z}(n))\otimes_{\mathbb{Z}} \mathrm{det}_{\mathbb{Z}}R\Gamma_{dR}(\mathcal{X}/\mathbb{Z})/F^n$$
should be a well defined invertible $\mathbb{Z}$-module endowed with a canonical trivialization
$$\mathbb{R}\stackrel{\sim}{\longrightarrow}\Delta(\mathcal{X}/\mathbb{Z},n)\otimes_{\mathbb{Z}}\mathbb{R}.$$
involving a fundamental class $\theta\in H^1(\mathbb{R},\mathbb{R})="H^1(W_{\mathbb{F}_1},\mathbb{R})"$ analogous to $e\in H^1(W_{\mathbb{F}_q},\mathbb{Z})$. Here $R\Gamma_{W,c}(\mathcal{X},\mathbb{Z}(n))$ is Weil-\'etale cohomology with compact support, where "compact support" sould be understood in the sense of \cite{Morin14}.
However, there is no natural trivialization  $\mathbb{R}\stackrel{\sim}{\rightarrow}\mathrm{det}_{\mathbb{Z}}R\Gamma_{W,c}(\mathcal{X},\mathbb{Z}(n))\otimes_{\mathbb{Z}} \mathbb{R}$. Consequently, it is not possible to define an Euler characteristic generalizing $\chi(H_W^*(X,\mathbb{Z}(n)),\cup e)$, neither to define a correcting factor generalizing Milne's correcting factor: one is forced to consider the fundamental line as a whole (we should point out that this viewpoint in conflict with \cite{Lichtenbaum09}). Let us go back to the case of smooth projective varieties $X/\mathbb{F}_q$, which we now see as schemes over $\mathbb{Z}$. Accordingly, we replace $Z(X,t)$ with $\zeta(X,s)=Z(X,q^{-s})$, the fundamental class $e$ with $\theta$ and the cotangent sheaf $\Omega^1_{X/\mathbb{F}_q}\simeq L_{X/\mathbb{F}_q}$ with the cotangent complex $L_{X/\mathbb{Z}}$. Assuming that $H^i_W(X,\mathbb{Z}(n))$ is finitely generated for all $i$, the fundamental line
$$\Delta(X/\mathbb{Z},n):=\mathrm{det}_{\mathbb{Z}}R\Gamma_{W}(X,\mathbb{Z}(n))\otimes_{\mathbb{Z}} \mathrm{det}_{\mathbb{Z}}R\Gamma(X,L\Omega^*_{X/\mathbb{Z}}/F^n)$$
is well defined and cup-product with $\theta$ gives a trivialization
$$\lambda:\mathbb{R}\stackrel{\sim}{\longrightarrow}\Delta(X/\mathbb{Z},n)\otimes_{\mathbb{Z}}\mathbb{R}.$$ 
Here $L\Omega^*_{X/\mathbb{Z}}/F^n$ is the derived de Rham complex modulo the Hodge filtration (see \cite{Illusie72} VIII.2.1). The aim of this note is to show that the Euler characteristic of $R\Gamma(X,L\Omega^*_{X/\mathbb{Z}}/F^n)$ equals $q^{\chi(X/\mathbb{F}_q,\mathcal{O}_X,n)}$, hence that Milne's correcting factor is naturally part of the fundamental line. We denote by $\zeta^*(X,n)$ the leading coefficient in the Taylor development of $\zeta(X,s)$ near $s=n$. 
\begin{thm*} Let $X$ be a smooth proper scheme over $\mathbb{F}_q$. Then we have
\begin{equation*}\label{part1}\Prod_{i\in\mathbb{Z}} \mid H^i(X,L\Omega^*_{X/\mathbb{Z}}/F^n)\mid^{(-1)^i}\,\,\, = \,\,\, q^{\chi(X/\mathbb{F}_q,\mathcal{O}_X,n)}.
\end{equation*}
Assume moreover that $X$ is projective and that the groups $H^i_W(X,\mathbb{Z}(n))$ are finitely generated for all $i$. Then one has 
\begin{eqnarray*}
\Delta(X/\mathbb{Z},n)&=&\lambda\left(\mathrm{log}(q)^{\rho_n}\cdot  \chi(H_W^*(X,\mathbb{Z}(n)),\cup e)^{-1}\cdot q^{-\chi(X/\mathbb{F}_q,\mathcal{O}_X,n)}\right)\cdot\mathbb{Z}\\
&=&\lambda\left(\zeta^*(X,n)^{-1}\right)\cdot\mathbb{Z}
\end{eqnarray*}
where $\rho_n:=-\mathrm{ord}_{s=n}\zeta(X,s)$ is the order of the pole of $\zeta(X,s)$ at $s=n$.
\end{thm*}
Before giving the proof, we need to fix some notations. For an object $C$ in the derived category of abelian groups such that $H^i(C)$ is finitely generated for all $i$ and $H^i(C)=0$ for almost all $i$, we set
$$\mathrm{det}_{\mathbb{Z}}(C):=\bigotimes_{i\in\mathbb{Z}}\mathrm{det}^{(-1)^i}_{\mathbb{Z}}H^i(C).$$ If  $H^i(C)$ is moreover finite for all $i$, then we call the following isomorphism
$$\mathrm{det}_{\mathbb{Z}}(C)\otimes_{\mathbb{Z}}\mathbb{Q}
\stackrel{\sim}{\rightarrow}\bigotimes_{i\in\mathbb{Z}}\mathrm{det}^{(-1)^i}_{\mathbb{Q}}\left(H^i(C)\otimes_{\mathbb{Z}}\mathbb{Q}\right)\stackrel{\sim}{\rightarrow}\bigotimes_{i\in\mathbb{Z}}\mathrm{det}^{(-1)^i}_{\mathbb{Q}}(0)\stackrel{\sim}{\rightarrow}\mathbb{Q}$$
the \emph{canonical $\mathbb{Q}$-trivialization} of $\mathrm{det}_{\mathbb{Z}}(C)$. If $A$ is a finite abelian group, we see $A$ as a complex concentrated in degree $0$ and use the same terminology for $\mathrm{det}_{\mathbb{Z}}(A)\otimes_{\mathbb{Z}}\mathbb{Q}\simeq\mathbb{Q}$. Finally, the notation $R\Gamma(X,-)$ refers to hypercohomology with respect to the Zariski topology.

\begin{proof} Firstly we prove the first assertion of the Theorem. Since Milne's correcting factor is insensitive to restriction of scalars (i.e. $q^{\chi(X/\mathbb{F}_q,\mathcal{O}_X,n)}=p^{\chi(X/\mathbb{F}_p,\mathcal{O}_X,n)}$), we may consider $X$ over $\mathbb{F}_p$. We need the  following
\begin{lem}\label{lemSS}
Let $E_*^{*,*}=(E_r^{p,q},d_r^{p,q})_r^{p,q}$ be a cohomological spectral sequence of abelian groups with abutment $H^*$. Assume that there exists an index $r_0$ such that $E_{r_0}^{p,q}$ is finite for all $(p,q)\in\mathbb{Z}^2$ and $E_{r_0}^{p,q}=0$ for all but finitely many $(p,q)$.   Then we have a canonical isomorphism
\begin{equation*}
\iota:\bigotimes_{p,q} \mathrm{det}^{(-1)^{p+q}}_{\mathbb{Z}}E_{r_0}^{p,q}\stackrel{\sim}{\longrightarrow}\bigotimes_{n} \mathrm{det}^{(-1)^{n}}_{\mathbb{Z}}H^{n}
\end{equation*}
such that the square of isomorphisms
\[ \xymatrix{
\left(\bigotimes_{p,q} \mathrm{det}^{(-1)^{p+q}}_{\mathbb{Z}}E_{r_0}^{p,q}\right)\otimes\mathbb{Q}\ar[r]^{\,\,\,\,\,\,\iota\otimes\mathbb{Q}}\ar[d]&
\ar[d]^{}\left(\bigotimes_{n}\mathrm{det}^{(-1)^{n}}_{\mathbb{Z}}H^{n}\right)\otimes\mathbb{Q}\\
\mathbb{Q}\ar[r]^{\mathrm{Id}}&\mathbb{Q}
}
\]
commutes, where the vertical maps are the canonical $\mathbb{Q}$-trivializations.
\end{lem}
\begin{proof} For any $t\geq r_0$, consider the bounded cochain complex $C_t^*$ of finite abelian groups:
$$...\longrightarrow \bigoplus_{p+q=n-1} E_t^{p,q}\longrightarrow \bigoplus_{p+q=n} E_t^{p,q}\stackrel{\oplus d_t^{p,q}}{\longrightarrow} \bigoplus_{p+q=n+1} E_t^{p,q}\longrightarrow ....$$
The fact that the cohomology of $C_t^*$ is given
by $H^n(C_t^*)= \bigoplus_{p+q=n} E_{t+1}^{p,q}$ gives an isomorphism
$$\bigotimes_{p,q} \mathrm{det}^{(-1)^{p+q}}_{\mathbb{Z}}E_{t}^{p,q}\stackrel{\sim}{\longrightarrow}\bigotimes_{p,q} \mathrm{det}^{(-1)^{p+q}}_{\mathbb{Z}}E_{t+1}^{p,q}$$
compatible with the canonical $\mathbb{Q}$-trivializations. By assumption, there exists an index $r_1\geq r_0$ such that the spectral sequence degenerates at the $r_1$-page, i.e. $E^{*,*}_{r_1}=E^{*,*}_{\infty}$. The induced filtration on each $H^n$ is such that
$\mathrm{gr}^p H^n=E_{\infty}^{p,n-p}$. We obtain isomorphisms
$$\bigotimes_{p,q} \mathrm{det}^{(-1)^{p+q}}_{\mathbb{Z}}E_{r_0}^{p,q}\stackrel{\sim}{\rightarrow}\bigotimes_{p,q} \mathrm{det}^{(-1)^{p+q}}_{\mathbb{Z}}E_{\infty}^{p,q}\stackrel{\sim}{\rightarrow}
\bigotimes_{n} \bigotimes_{p} \mathrm{det}^{(-1)^{n}}_{\mathbb{Z}}E_{\infty}^{p,n-p}\stackrel{\sim}{\rightarrow}\bigotimes_{n} \mathrm{det}^{(-1)^{n}}_{\mathbb{Z}}H^{n}$$
compatible with the canonical $\mathbb{Q}$-trivializations.
\end{proof}
Consider the Hodge filtration $F^*$  on the derived de Rham complex $L\Omega^*_{X/\mathbb{Z}}$. By (\cite{Illusie72} VIII.2.1.1.5) we have 
$$\mathrm{gr}(L\Omega^*_{X/\mathbb{Z}})\simeq \bigoplus_{p\geq 0} L\Lambda^pL_{X/\mathbb{Z}}[-p].$$
This gives a (convergent) spectral sequence
$$E_1^{p,q}=H^q(X,L\Lambda^{p<n}L_{X/\mathbb{Z}})\Longrightarrow
H^{p+q}(X,L\Omega^*_{X/\mathbb{Z}}/F^n)$$
where $L\Lambda^{p<n}L_{X/\mathbb{Z}}:=L\Lambda^{p}L_{X/\mathbb{Z}}$ for $p<n$ and $L\Lambda^{p<n}L_{X/\mathbb{Z}}:=0$ otherwise. The scheme $X$ is proper and $L\Lambda^pL_{X/\mathbb{Z}}$ is isomorphic, in the derived category  $\mathcal{D}(\mathcal{O}_X)$ of $\mathcal{O}_X$-modules, to a bounded complex of coherent sheaves, so that $E_1^{p,q}$ is a finite dimensional $\mathbb{F}_p$-vector space for all $(p,q)$ vanishing for almost all $(p,q)$. By Lemma \ref{lemSS}, this yields
isomorphisms
\begin{eqnarray*}
\mathrm{det}_{\mathbb{Z}}R\Gamma(X,L\Omega^*_{X/\mathbb{Z}}/F^n)&\stackrel{\sim}{\longrightarrow}&\bigotimes_{i}\mathrm{det}^{(-1)^i}_{\mathbb{Z}}H^i(X,L\Omega^*_{X/\mathbb{Z}}/F^n)\\
&\stackrel{\sim}{\longrightarrow}&\bigotimes_{p<n,q}\mathrm{det}^{(-1)^{p+q}}_{\mathbb{Z}}H^q(X,L\Lambda^pL_{X/\mathbb{Z}})\\
&\stackrel{\sim}{\longrightarrow}&\bigotimes_{p<n}\mathrm{det}^{(-1)^{p}}_{\mathbb{Z}}R\Gamma(X,L\Lambda^pL_{X/\mathbb{Z}})
\end{eqnarray*}
which are compatible with the canonical  $\mathbb{Q}$-trivializations. The transitivity triangle (see \cite{Illusie71} II.2.1) for the composite map $X\stackrel{f}{\rightarrow} \mathrm{Spec}(\mathbb{F}_p)\rightarrow \mathrm{Spec}(\mathbb{Z})$ reads as follows (using \cite{Illusie71} III.3.1.2 and \cite{Illusie71} III.3.2.4(iii)):
\begin{equation}\label{etr}
Lf^*(p\mathbb{Z}/p^2\mathbb{Z})[1]\rightarrow L_{X/\mathbb{Z}}\rightarrow \Omega^1_{X/\mathbb{F}_p}[0] \stackrel{\omega}{\rightarrow} Lf^*(p\mathbb{Z}/p^2\mathbb{Z})[2].
\end{equation}
We set $\mathcal{L}:=Lf^*(p\mathbb{Z}/p^2\mathbb{Z})$, a trivial invertible $\mathcal{O}_{X}$-module. By (\cite{Illusie71} Th\'eor\`eme III.2.1.7), the class 
$$\omega\in \mathrm{Ext}^2_{\mathcal{O}_{X}}(\Omega^1_{X/\mathbb{F}_p},\mathcal{L})\simeq H^2(X,T_{X/\mathbb{F}_p})$$ is the obstruction to the existence of a lifting of $X$ over $\mathbb{Z}/p^2\mathbb{Z}$. If such a lifting does exist then we have $\omega=0$, in which case the following lemma is superfluous. For an object $C$ of $D(\mathcal{O}_X)$ with bounded cohomology, we set $$\mathrm{gr}_{\tau}C:=\bigoplus_{i\in \mathbb{Z}} H^i(C)[-i].$$
\begin{lem}
We have an isomorphism
$$\mathrm{det}_{\mathbb{Z}}R\Gamma(X, L\Lambda^p L_{X/\mathbb{Z}})\simeq \mathrm{det}_{\mathbb{Z}}R\Gamma(X, L\Lambda^p (\mathrm{gr}_{\tau}L_{X/\mathbb{Z}}))$$
compatible with the canonical $\mathbb{Q}$-trivializations.
\end{lem}

\begin{proof} 
The map $X\rightarrow \mathrm{Spec}(\mathbb{Z})$ is a local complete intersection, hence the complex $L_{X/\mathbb{Z}}$ has perfect amplitude $\subset[-1,0]$ (see \cite{Illusie71} III.3.2.6). In other words, $L_{X/\mathbb{Z}}$ is locally isomorphic in $\mathcal{D}(\mathcal{O}_X)$ to a complex of free modules of finite type concentrated in degrees $-1$ and $0$. By (\cite{SGA6II} 2.2.7.1) and (\cite{SGA6II} 2.2.8), $L_{X/\mathbb{Z}}$ is globally isomorphic to such a complex, i.e. there exists an isomorphism $L_{X/\mathbb{Z}}\simeq [M\rightarrow N]$ in $\mathcal{D}(\mathcal{O}_X)$, where $M$ and $N$ are finitely generated locally free $\mathcal{O}_X$-modules put in degrees $-1$ and $0$ respectively. Consider the exact sequences
\begin{equation}\label{2ES}
0\rightarrow \mathcal{L}\rightarrow M\rightarrow F\rightarrow 0\mbox{ and } 0\rightarrow F\rightarrow N\rightarrow \Omega\rightarrow 0
\end{equation}
where $\mathcal{L}:=Lf^*(p\mathbb{Z}/p^2\mathbb{Z})$ and $\Omega:= \Omega^1_{X/\mathbb{F}_p}$ are finitely generated and locally free. It follows that $F$ is also finitely generated and locally free. One has an isomorphism in $\mathcal{D}(\mathcal{O}_X)$
\begin{equation}\label{QI}
L\Lambda^pL_{X/\mathbb{Z}}\simeq [\Gamma^pM\rightarrow \Gamma^{p-1}M\otimes N\rightarrow ... \rightarrow
M\otimes\Lambda^{p-1} N\rightarrow \Lambda^{p} N]
\end{equation}
where the right hand side is concentrated in degrees $[-p,0]$  (see \cite{Illusie72} VIII.2.1.2 and \cite{Illusie71} I.4.3.2.1). Moreover, in view of (\ref{etr}) we may choose an isomorphism
$$\mathrm{gr}_{\tau}L_{X/\mathbb{Z}}\simeq [\mathcal{L}\stackrel{0}{\rightarrow}\Omega]$$
in $\mathcal{D}(\mathcal{O}_{X})$, the right hand side being concentrated in degrees $[-1,0]$. Hence the complex $L\Lambda^p(\mathrm{gr}_{\tau}L_{X/\mathbb{Z}})\in\mathcal{D}(\mathcal{O}_{X})$ is represented by a complex of the form
\begin{equation}\label{QI2}
L\Lambda^p(\mathrm{gr}_{\tau}L_{X/\mathbb{Z}})\simeq L\Lambda^p([\mathcal{L}\rightarrow\Omega])
\simeq [\Gamma^p\mathcal{L}\rightarrow \Gamma^{p-1}\mathcal{L}\otimes \Omega\rightarrow ... \rightarrow
\mathcal{L}\otimes\Lambda^{p-1} \Omega\rightarrow \Lambda^{p} \Omega]
\end{equation}
concentrated in degrees $[-p,0]$. Lemma \ref{lemSS} and (\ref{QI}) give an isomorphism
\begin{equation}\label{IS0}
\mathrm{det}_{\mathbb{Z}}R\Gamma(X, L\Lambda^p L_{X/\mathbb{Z}})\simeq \bigotimes_{0\leq q\leq p} \mathrm{det}^{(-1)^{p-q}}_{\mathbb{Z}}R\Gamma(X,
\Gamma^{p-q}M\otimes \Lambda^q N)
\end{equation}
compatible with the $\mathbb{Q}$-trivializations.
The second exact sequence in (\ref{2ES}) endows $\Lambda^{q}N$ with a finite decreasing filtration $\mathrm{Fil}^*$ such that $\mathrm{gr}_{\mathrm{Fil}}^i(\Lambda^{q}N)=\Lambda^{i}F\otimes \Lambda^{q-i}\Omega$. Since $\Gamma^{p-q}M$ is flat, $\mathrm{Fil}^*$ induces a similar filtration on $\Gamma^{p-q}M\otimes\Lambda^{q}N$ such that
$$\mathrm{gr}_{\mathrm{Fil}}^i(\Gamma^{p-q}M\otimes \Lambda^{q}N)=\Gamma^{p-q}M\otimes \Lambda^{i}F\otimes \Lambda^{q-i}\Omega.$$
This filtration induces an isomorphism 
\begin{equation}\label{IS}
\mathrm{det}_{\mathbb{Z}}R\Gamma(X,\Gamma^{p-q}M\otimes\Lambda^{q}N)\simeq
\bigotimes_{0\leq i\leq q} \mathrm{det}_{\mathbb{Z}}R\Gamma(X,\Gamma^{p-q}M\otimes \Lambda^{i}F\otimes \Lambda^{q-i}\Omega)
\end{equation}
compatible with the $\mathbb{Q}$-trivializations. Lemma \ref{lemSS} and (\ref{QI2}) give an isomorphism
\begin{equation}\label{IS2}
\mathrm{det}_{\mathbb{Z}}R\Gamma(X,L\Lambda^p(\mathrm{gr}_{\tau}L_{X/\mathbb{Z}}))\simeq
\bigotimes_{0\leq i\leq p} \mathrm{det}^{(-1)^{p-i}}_{\mathbb{Z}}R\Gamma(X,\Gamma^{p-i}\mathcal{L}\otimes \Lambda^i\Omega)
\end{equation}
compatible with the $\mathbb{Q}$-trivializations. Moreover, we have an isomorphism  (see \cite{Illusie71} I.4.3.1.7)
$$\Gamma^{p-i}\mathcal{L}\simeq[\Gamma^{p-i}M\rightarrow
\Gamma^{p-i-1}M\otimes F \rightarrow ... \rightarrow
M\otimes \Lambda^{p-i-1}F \rightarrow
\Lambda^{p-i}F]$$
where the right hand side sits in degrees $[0,p-i]$. Since $\Lambda^i\Omega$ is flat, we have an isomorphism between 
$\Gamma^{p-i}\mathcal{L}\otimes \Lambda^i\Omega$ and 
$$[\Gamma^{p-i}M \otimes \Lambda^i\Omega\rightarrow
\Gamma^{p-i-1}M\otimes F \otimes \Lambda^i\Omega \rightarrow ... \rightarrow
M\otimes \Lambda^{p-i-1}F \otimes \Lambda^i\Omega\rightarrow
\Lambda^{p-i}F\otimes \Lambda^i\Omega].$$
By Lemma \ref{lemSS}, we have
\begin{equation}\label{IS3}
\mathrm{det}_{\mathbb{Z}}R\Gamma(X,\Gamma^{p-i}\mathcal{L}\otimes \Lambda^i\Omega)
\simeq
\bigotimes_{0\leq j\leq p-i} \mathrm{det}^{(-1)^{j}}_{\mathbb{Z}}R\Gamma(X,
\Gamma^{p-i-j}M\otimes \Lambda^j F \otimes \Lambda^i\Omega).
\end{equation}
Putting (\ref{IS2}), (\ref{IS3}), (\ref{IS})  and  (\ref{IS0}) together, we obtain isomorphisms
\begin{eqnarray*}
\mathrm{det}_{\mathbb{Z}}R\Gamma(X, L\Lambda^p(\mathrm{gr}_{\tau}L_{X/\mathbb{Z}}))
&\simeq&\bigotimes_{0\leq i\leq p} \mathrm{det}^{(-1)^{p-i}}_{\mathbb{Z}}R\Gamma(X,\Gamma^{p-i}\mathcal{L}\otimes \Lambda^i\Omega)\\
&\simeq&\bigotimes_{0\leq i\leq p}\left(\bigotimes_{0\leq j\leq p-i} \mathrm{det}^{(-1)^{p-i-j}}_{\mathbb{Z}}R\Gamma(X,\Gamma^{p-i-j}M\otimes \Lambda^j F \otimes \Lambda^i\Omega)\right)\\
&=&
\bigotimes_{0\leq q\leq p}\left(\bigotimes_{0\leq i, j\,\, ;\, i+j=q} \mathrm{det}^{(-1)^{p-q}}_{\mathbb{Z}}R\Gamma(X,
\Gamma^{p-q}M\otimes \Lambda^j F \otimes \Lambda^i\Omega)\right)\\
&\simeq& 
\bigotimes_{0\leq q\leq p} \mathrm{det}^{(-1)^{p-q}}_{\mathbb{Z}}R\Gamma(X,
\Gamma^{p-q}M\otimes \Lambda^q N)\\
&\simeq& \mathrm{det}_{\mathbb{Z}}R\Gamma(X, L\Lambda^p L_{X/\mathbb{Z}})
\end{eqnarray*}
compatible with the canonical $\mathbb{Q}$-trivializations.
\end{proof}
Recall from (\ref{QI2}) that the complex $L\Lambda^p(\mathrm{gr}_{\tau}L_{X/\mathbb{Z}})$ is isomorphic in $\mathcal{D}(\mathcal{O}_{X})$ to a complex of the form
$$0\rightarrow \Gamma^p\mathcal{L}\rightarrow \Gamma^{p-1}\mathcal{L}\otimes \Omega^1_{X/\mathbb{F}_p}\rightarrow ... \rightarrow \Gamma^{1}\mathcal{L}\otimes \Omega^{p-1}_{X/\mathbb{F}_p}\rightarrow \Omega^{p}_{X/\mathbb{F}_p}\rightarrow 0$$
concentrated in degrees $[-p,0]$. Choose an isomorphism of $\mathbb{F}_p$-vector spaces $\mathbb{F}_p\simeq p\mathbb{Z}/p^2\mathbb{Z}$. This induces an identification $\mathcal{L}\simeq \mathcal{O}_{X}$, and more generally
$$\Gamma^i\mathcal{L}\simeq S^i\mathcal{L}\simeq \mathcal{O}_{X}$$
for any $i\geq 0$. Hence $(L\Lambda^p(\mathrm{gr}_{\tau}L_{X/\mathbb{Z}}))[-p]\in\mathcal{D}(\mathcal{O}_{X})$ is represented by a complex of the form
\begin{equation}\label{here}
0\rightarrow \mathcal{O}_{X}\rightarrow \Omega^1_{X/\mathbb{F}_p}\rightarrow ... \rightarrow \Omega^{p}_{X/\mathbb{F}_p}\rightarrow 0
\end{equation}
concentrated in degrees $[0,p]$. Hence there is a spectral sequence of the form 
$$E_1^{i,j}=H^j(X,\Omega^{i\leq p}_{X/\mathbb{F}_p})\Longrightarrow H^{i+j}(X,(L\Lambda^p(\mathrm{gr}_{\tau}L_{X/\mathbb{Z}}))[-p])$$
where $\Omega^{i\leq p}:=\Omega^{i}$ for $i\leq p$ and $\Omega^{i\leq p}:=0$ for $i> p$. By Lemma \ref{lemSS} again, we get an identification
\begin{eqnarray*}
  \bigotimes_{i\leq p,j}\mathrm{det}_{\mathbb{Z}}^{(-1)^{i+j}} H^{j}(X,\Omega^i_{X/\mathbb{F}_p})&\stackrel{\sim}{\longrightarrow}&\mathrm{det}_{\mathbb{Z}}R\Gamma(X,(L\Lambda^p(\mathrm{gr}_{\tau}L_{X/\mathbb{Z}}))[-p])\\
&\stackrel{\sim}{\longrightarrow}&\mathrm{det}_{\mathbb{Z}}^{(-1)^{p}} R\Gamma(X,L\Lambda^p(\mathrm{gr}_{\tau}L_{X/\mathbb{Z}})).
\end{eqnarray*}
In summary, we have the following isomorphisms
\begin{eqnarray}
\mathrm{det}_{\mathbb{Z}}R\Gamma(X,L\Omega^*_{X/\mathbb{Z}}/F^n)
&\stackrel{\sim}{\longrightarrow}&\bigotimes_{p<n}\mathrm{det}^{(-1)^{p}}_{\mathbb{Z}}R\Gamma(X,L\Lambda^pL_{X/\mathbb{Z}})\\
&\stackrel{\sim}{\longrightarrow}&\bigotimes_{p<n}\mathrm{det}^{(-1)^{p}}_{\mathbb{Z}}R\Gamma(X,L\Lambda^p(\mathrm{gr}_{\tau}L_{X/\mathbb{Z}}))\\
\label{last}&\stackrel{\sim}{\longrightarrow}&\bigotimes_{p<n}\left(\bigotimes_{i\leq p,j}\mathrm{det}_{\mathbb{Z}}^{(-1)^{i+j}} H^{j}(X,\Omega^i_{X/\mathbb{F}_p})\right)
\end{eqnarray}
such that the square
\[ \xymatrix{
\left(\mathrm{det}_{\mathbb{Z}}R\Gamma(X,L\Omega^*_{X/\mathbb{Z}}/F^n)\right)\otimes\mathbb{Q}\ar[r]\ar[d]^{\gamma}&
\ar[d]^{\gamma'}\left(\bigotimes_{p<n}\bigotimes_{i\leq p,j}\mathrm{det}^{(-1)^{i+j}}_{\mathbb{Z}}H^{j}(X,\Omega^i_{X/\mathbb{F}_p})\right)\otimes\mathbb{Q}\\
\mathbb{Q}\ar[r]^{\mathrm{Id}}&\mathbb{Q}
}
\]
commutes, where the top horizontal map is induced by (\ref{last}), and the vertical isomorphisms are the canonical trivializations. The first assertion of the theorem follows:
\begin{eqnarray*}
\left(\Prod_{i\in\mathbb{Z}} \mid H^i(X,L\Omega^*_{X/\mathbb{Z}}/F^n)\mid^{(-1)^i}\right)^{-1}\cdot\mathbb{Z}
&=&\gamma\left(\mathrm{det}_{\mathbb{Z}}R\Gamma(X,L\Omega^*_{X/\mathbb{Z}}/F^n)\right)\\
&=&\gamma'\left(\bigotimes_{p<n} \bigotimes_{i\leq p,j}\mathrm{det}^{(-1)^{i+j}}_{\mathbb{Z}}H^{j}(X,\Omega^i_{X/\mathbb{F}_p})\right)\\
&=&p^{-\chi(X/\mathbb{F}_p,\mathcal{O}_{X},n)}\cdot \mathbb{Z}.
\end{eqnarray*}\\

We now explain why the second assertion of the theorem is a restatement of (\cite{Geisser-Weiletale} Theorem 1.3). We assume that $H^i(X,\mathbb{Z}(n))$ is finitely generated for all $i\in\mathbb{Z}$ ($X$ and $n$ being fixed). It follows that $H^i(X,\mathbb{Z}(n))$ is in fact finite for $i\neq 2n,2n+1$ and that the complex (\ref{LG-complex}) has finite cohomology groups. In particular the complex
\begin{equation}\label{acyclicwithQ}
...\stackrel{\cup e}{\longrightarrow} H^i_W(X,\mathbb{Z}(n))\otimes\mathbb{Q} \stackrel{\cup e}{\longrightarrow} H^{i+1}_W(X,\mathbb{Z}(n))\otimes\mathbb{Q} \stackrel{\cup e}{\longrightarrow}... 
\end{equation}
is acyclic and concentrated in degrees $[2n,2n+1]$. It also follows from the finite generation of the Weil-\'etale cohomology groups that 
$$\mathrm{ord}_{s=n}\zeta(X,s)= -\mathrm{rank}_{\mathbb{Z}}H^i_W(X,\mathbb{Z}(n))=:-\rho_n.$$
The acyclic complex (\ref{acyclicwithQ})
induces a trivialization
$$\beta:\mathbb{Q}\stackrel{\sim}{\longrightarrow} \bigotimes_{i} \mathrm{det}^{(-1)^{i}}_{\mathbb{Q}} \left( H^i_W(X,\mathbb{Z}(n))\otimes\mathbb{Q}\right)\stackrel{\sim}{\longrightarrow}
\left(\bigotimes_{i} \mathrm{det}^{(-1)^{i}}_{\mathbb{Z}} H^i_W(X,\mathbb{Z}(n))\right)\otimes\mathbb{Q}$$
such that 
$$\beta\left(\chi(H_W^*(X,\mathbb{Z}(n)),\cup e)^{-1}\right)\cdot\mathbb{Z}=\bigotimes_{i} \mathrm{det}^{(-1)^i}_{\mathbb{Z}} H^i_W(X,\mathbb{Z}(n)).$$
The class $e\in H^1(W_{\mathbb{F}_q},\mathbb{Z})=\mathrm{Hom}(W_{\mathbb{F}_q},\mathbb{Z})$ maps the Frobenius $\mathrm{Frob}\in W_{\mathbb{F}_q}$ to $1\in\mathbb{Z}$. We define the map $$W_{\mathbb{F}_q}=\mathbb{Z}\cdot\mathrm{Frob}\longrightarrow\mathbb{R}=:W_{\mathbb{F}_1}$$
as the map sending $\mathrm{Frob}$ to $\mathrm{log}(q)$, while $\theta\in H^1(W_{\mathbb{F}_1},\mathbb{R})=\mathrm{Hom}(\mathbb{R},\mathbb{R})$ is the identity map. It follows that the acyclic complex
$$...\stackrel{\cup \theta}{\longrightarrow} H^i_W(X,\mathbb{Z}(n))\otimes\mathbb{R} \stackrel{\cup \theta}{\longrightarrow} H^{i+1}_W(X,\mathbb{Z}(n))\otimes\mathbb{R} \stackrel{\cup \theta}{\longrightarrow}... $$
induces a trivialization
$$\alpha:\mathbb{R}\stackrel{\sim}{\longrightarrow} \bigotimes_{i} \mathrm{det}^{(-1)^{i}}_{\mathbb{R}} \left( H^i_W(X,\mathbb{Z}(n))\otimes\mathbb{R}\right)\stackrel{\sim}{\longrightarrow}
\left(\bigotimes_{i} \mathrm{det}^{(-1)^{i}}_{\mathbb{Z}} H^i_W(X,\mathbb{Z}(n))\right)\otimes\mathbb{R}$$
such that 
$$\alpha\left(\chi(H_W^*(X,\mathbb{Z}(n)),\cup e)^{-1}\cdot \mathrm{log}(q)^{\rho_n}\right)\cdot\mathbb{Z}=\bigotimes_{i} \mathrm{det}^{(-1)^i}_{\mathbb{Z}} H^i_W(X,\mathbb{Z}(n)).$$ 
The trivialization $\lambda$ is the product of $\alpha$ with the canonical trivialization
$$\mathbb{R}\stackrel{\sim}{\longrightarrow}\mathrm{det}_{\mathbb{Z}}R\Gamma(X,L\Omega^*_{X/\mathbb{Z}}/F^n)\otimes_{\mathbb{Z}}\mathbb{R}.$$
Hence we have
$$\lambda\left(\mathrm{log}(q)^{\rho_n}\cdot  \chi(H_W^*(X,\mathbb{Z}(n)),\cup e)^{-1}\cdot q^{-\chi(X,\mathcal{O}_X,n)}\right)\cdot\mathbb{Z}=\Delta(X/\mathbb{Z},n).$$
Moreover, formula (\ref{LG-Conj}) gives
$$\zeta^*(X,s)= \pm\mathrm{log}(q)^{-\rho_n} \cdot \chi(H_W^*(X,\mathbb{Z}(n)),\cup e) \cdot q^{\chi(X,\mathcal{O}_X,n)}$$
hence the result follows from (\cite{Geisser-Weiletale} Theorem 1.3).
\end{proof}


\begin{thebibliography}{6}
\bibitem{Geisser-Weiletale}{Geisser, T.: \emph{Weil-\'etale cohomology over finite fields.}  Math. Ann.  330 (4) (2004), 665$-$692.}
\bibitem{Geisser-arith-coh}{Geisser, T.: \emph{Arithmetic cohomology over finite fields and special valyues of $\zeta$-functions.} Duke Math. J. 133 (1) (2006), 27--57.}
\bibitem{SGA6II}{Illusie, L.: \emph{Existence de R\'esolutions Globales.} Th\'eorie des Intersections et Th\'eor\`eme de Riemann-Roch (SGA6), Lecture Notes in Mathematics 225, (1971), 160--221.}
\bibitem{Illusie71}{Illusie, L.: \emph{Complexe cotangent et d\'eformations. I.} Lecture Notes in Mathematics, Vol. 239. Springer-Verlag, Berlin-New York, 1971.}
\bibitem{Illusie72}{Illusie, L.: \emph{Complexe cotangent et d\'eformations. II.} Lecture Notes in Mathematics, Vol. 283. Springer-Verlag, Berlin-New York, 1972.}
\bibitem{Lichtenbaum-finite-field}{Lichtenbaum, S.: \emph{The Weil-\'etale topology on schemes over finite fields.}
Compositio Math.  141 (3) (2005), 689$-$702.}
\bibitem{Lichtenbaum09}{Lichtenbaum, S.: \emph{Euler characteristics and special values of zeta-functions.} Motives and algebraic cycles, 249--255, Fields Inst. Commun., 56, Amer. Math. Soc., Providence, RI, 2009.}
\bibitem{Milne86}{Milne, J.S.: \emph{Values of zeta functions of varieties over finite fields.} Amer. J. Math. 108 (1986), no. 2, 297--360.}
\bibitem{Milne13}{Milne, J.S.: \emph{Addendum to: Values of zeta functions of varieties over
finite fields, Amer. J. Math. 108, (1986), 297-360.} (2013), available at http://www.jmilne.org/.}
\bibitem{Milne-Ramachandran13}{Milne, J.S., Ramachandran, N.: \emph{Motivic complexes and special values of zeta functions.} Preprint 2013,  arXiv:1311.3166.}

\bibitem{Morin14}{Morin, B.: \emph{Zeta functions of regular arithmetic schemes at $s=0$}. Duke Math. J. 163 (2014), no. 7, 1263--1336.}

\end{thebibliography}
\end{document}